\documentclass[12pt]{article}
\usepackage{amsmath}
\usepackage{amsfonts}
\usepackage{amsthm}
\usepackage{graphicx}
\graphicspath{{Figures/}} 
\usepackage{overpic}
\usepackage{amssymb} 
\usepackage{pictex}
\usepackage{rotating}  
\usepackage{xcolor}
\usepackage{cite} 

\usepackage{enumitem} 

\usepackage{accents} 

\usepackage{yhmath} 
\usepackage{etex} 

\usepackage{comment}

\definecolor{refkey}{rgb}{0,0,1}
\definecolor{labelkey}{rgb}{0,0,1}
\numberwithin{equation}{section}

\newtheorem{theorem}{Theorem}[section]
\newtheorem{proposition}[theorem]{Proposition}
\newtheorem{lemma}[theorem]{Lemma}
\newtheorem{corollary}[theorem]{Corollary}
\newtheorem{Definition}[theorem]{Definition}
\newenvironment{definition}{\begin{Definition}\rm}{\end{Definition}}
\newtheorem{Remark}[theorem]{Remark}
\newenvironment{remark}{\begin{Remark}\rm}{\end{Remark}}
\newtheorem{RHproblem}[theorem]{RH problem}

\newtheorem{Example}[theorem]{Example}

\newcommand{\C}{\mathbb{C}}
\newcommand{\D}{\mathbb D}

\newcommand{\N}{\mathbb{N}}

\newcommand{\R}{\mathbb{R}}

\newcommand{\T}{\mathbb{T}}

\newcommand{\MM}{\mathcal M}

\newcommand{\OO}{\mathcal O}
\newcommand{\PP}{\mathcal P}

\newcommand{\eps}{\epsilon}

\newcommand{\vphi}{\varphi}

\renewcommand{\Re}{{\rm Re} \,}

\newcommand{\inter}{{\rm Int} }

\def\supp{\mathop{\mathrm{supp}}\nolimits}

\newcommand{\p}{\partial}
\newcommand{\Om}{\Omega}
\renewcommand{\bar}{\overline}
\renewcommand{\tilde}{\widetilde}
\renewcommand{\hat}{\widehat}

\usepackage{color}

\usepackage[bookmarksopen, naturalnames]{hyperref}
\hypersetup{
    colorlinks,
    linkcolor={blue},
    citecolor={blue},
}

\begin{document}
\title{
An extremal problem for the Bergman kernel\\
of orthogonal polynomials}

\author{S. Charpentier, N. Levenberg and F. Wielonsky}
\date{\today}
\maketitle
\begin{abstract} Let $\Gamma \subset \C$ be a curve of class $C(2,\alpha)$. For $z_{0}$ in the unbounded component of $\C\setminus \Gamma$, and for $n=1,2,...$, let $\nu_n$ be a probability measure with $\supp(\nu_{n})\subset \Gamma$ which minimizes the Bergman function $B_{n}(\nu,z):=\sum_{k=0}^{n}|q_{k}^{\nu}(z)|^{2}$ at $z_{0}$ among all probability measures $\nu$ on $\Gamma$ (here, $\{q_{0}^{\nu},\ldots,q_{n}^{\nu}\}$ are an orthonormal basis in $L^2(\nu)$ for the holomorphic polynomials of degree at most $n$). We show that $\{\nu_{n}\}_n$ tends weak-* to $\hat\delta_{z_{0}}$, the balayage of the point mass at $z_0$ onto $\Gamma$, by relating this to an optimization problem for probability measures on the unit circle. Our proof makes use of estimates for Faber polynomials associated to $\Gamma$.

\end{abstract}

\bigskip

\noindent Keywords: orthogonal polynomials, Bergman kernel, Faber polynomials, Szeg\H{o}  function

\medskip

\noindent MSC Classification Numbers: 30C40, 30E10

\bigskip

\section{Introduction}
Let $K$ be a compact subset of the complex plane $\C$ and let $\MM(K)$ denote the probability measures on $K$. Given a positive integer $n$, if the support of $\nu \in \MM(K)$ contains at least $n+1$ points, we can form the associated Bergman function 
$$B_{n}(\nu,z):=\sum_{k=0}^{n}|q_{k}^{\nu}(z)|^{2},$$ 
where  $\{q_{0}^{\nu},\ldots,q_{n}^{\nu}\}$ form an orthonormal basis in $L^2(\nu)$ for $\PP_n$, the holomorphic polynomials of degree at most $n$. One can set $B_{n}(\nu,z)=+\infty$ when $\nu$ has less than $n+1$ points in its support, but since we are interested in asymptotics ($n\to \infty$) of Bergman functions, we may assume $K$ contains infinitely many points. We fix $z_{0}\in\Omega$, the unbounded component of $\C\setminus K$, and for each $n$ we consider a probability measure $\nu_{n}$ with $\supp(\nu_{n})\subset K$ which minimizes the Bergman function at $z_{0}$ among all such $\nu \in \MM(K)$:
$$B_{n}(\nu_{n},z_{0})=\min_{\nu\in\MM(K)}B_{n}(\nu,z_{0}).$$
The existence of at least one minimizing measure follows from the weak-$*$ compactness of $\MM(K)$ and lower semicontinuity of the map $\nu\mapsto B_{n}(\nu,z_{0})$, see Lemma \ref{low-sc}.

Equivalently, $\nu_{n}$ solves the max-min problem
\begin{equation}\label{max-min}
\max_{\nu\in\MM(K)}\lambda_{n}(\nu,z_{0}),
\qquad
\lambda_{n}(\nu,z_{0})=\min_{p\in\PP_{n},~p(z_{0})=1}\int_{K}|p|^{2}d\nu\leq1,
\end{equation}
where $\lambda_{n}(\nu,z_{0})$ is the Christoffel function of $\nu$ at $z_{0}$.
We recall that (cf., \cite[Theorem 1.4]{VA})
\begin{equation}\label{B-lambda}
\lambda_{n}(\nu,z)=B_{n}(\nu,z)^{-1},\qquad n\geq0,
\end{equation}
where we note that, with our previous convention on $B_{n}(\nu,z)$, the equality still holds when $\nu$ has less than $n+1$ points.

Such an extremal measure $\nu_{n}$ is called an {\bf optimal prediction measure} (OPM) for $K$ and $z_{0}$ of order $n$. In general, it is not unique. For motivation to study this problem, we refer to \cite{BLOC} where they give a nice application to the field of optimal design for polynomial regression. Although 
$B_{n}(\nu,z)$ is well defined only if
all orthogonal polynomials up to degree $n$, exist, 
$\lambda_{n}(\nu,z)$ is always defined, equal to $0$ when the support of $\nu$ consists of fewer than $n+1$ points. In fact, $\lambda_{n}(\nu,z)$ is defined for all $z\in \C$. For 
an extremal measure $\nu_{n}$, all
the orthogonal polynomials $q_{k}^{\nu_{n}}$, $k=0,\ldots,n$, do exist.
Note also that, for each $n$, the Bergman function $B_{n}(\nu,z)$ only depends on a finite number of  moments of the measure $\nu$, namely
\begin{equation}\label{moments}
m_{j,k}=\int_{K}z^{j}\bar z^{k}d\nu,\quad j,k=0,\ldots,n.
\end{equation}
It is known that $B_{n}(\nu_{n},z_{0})$ is related to the polynomial of extremal growth at $z_0$, see \cite{BLOC}. Indeed, one has
\begin{equation}\label{max-growth}
B_{n}(\nu_{n},z_{0})=\sup_{p\in\PP_{n}}\frac{|p(z_{0})|^{2}}{\|p\|_{K}^{2}}\leq e^{2ng_{\Omega}(z_{0})},
\end{equation}
where the upper bound, with $g_{\Omega}$ the Green function of $\Omega$ and the point $\infty$, follows from the fact that
\begin{equation}\label{bw} g_{\Omega}(z_0)=\sup \{\frac{1}{deg(p)}\log |p(z_0)|: \ p\in \cup_n \PP_n, \ ||p||_K\leq 1\}. \end{equation} 
Here deg$(p)$ denotes the degree of $p$ and $||p||_K:=\sup_{z\in K}|p(z)|$. Note that polynomials of extremal growth are also studied in the recent paper \cite{CSZ} where they are called dual residual polynomials.

For a general probability measure $\nu$ on $K$ and $z\in\C$, we have that $$1\geq \lambda_{n}(\nu,z) \geq \lambda_{n+1}(\nu,z)\geq 0$$ 
so that the limit
\begin{equation}
\label{def-lamb-inf}
\lambda_{\infty}(\nu,z):=\lim_{n\to\infty}\lambda_{n}(\nu,z)
\end{equation}
exists and $0\leq \lambda_{\infty}(\nu,z)\leq 1$.
It has been verified by explicit computations in \cite{BLOC} that if: \\[5pt]
i) $K=[-1,1]$ and $z_{0}$ is real or purely imaginary, \\
ii) $K=\bar \D:=\{z\in \C: |z|\leq 1\}$ and $|z_{0}|>1$, \\[5pt]
certain sequences of optimal prediction measures $\nu_{n}$ tend weak-* to a limit, namely
\begin{equation}\label{conv}
\nu_{n}\to\hat\delta_{z_{0}},\qquad n\to\infty,
\end{equation}
where $\hat\delta_{z_{0}}$ denotes the balayage measure of $\delta_{z_{0}}$, the point mass at $z_0$, onto $K$. The authors of \cite{BLOC} have conjectured that this convergence holds true more generally. It is the aim of the present paper to show that the conjecture holds true for a more general class of compact sets $K$ and points $z_{0}$ outside of~$K$. Namely, our main result is the following theorem.
\begin{theorem} \label{keyth}
Assume $K$ is a compact subset bounded by a curve $\Gamma\in C(2,\alpha)$, $0<\alpha<1$ (i.e.,\ $\Gamma$ can be parameterized by a function of class $C(2,\alpha)$). For $z_0\in \Omega$, any sequence of optimal prediction measures $\{\nu_{n}\}_n$ tends weak-* to $\hat\delta_{z_{0}}$, the balayage of $\delta_{z_{0}}$ onto $\Gamma$.
\end{theorem}

\noindent Here, $C(k,\alpha)$ denotes the class of $k-$times continuously differentiable functions with $k-$th derivative satisfying a Lipschitz condition of order $\alpha$.

After some general preliminaries in the next section, in section 3 we complement the study in \cite{BLOC} of the case of $K=\bar\D$, the closed unit disk. We show in Theorem \ref{optim-D} that for $z\in \D$, the balayage $\hat\delta_{z}$ to $ \T :=\partial \D$ is the unique maximizer of $\lambda_{\infty}(\mu,z)$ from (\ref{def-lamb-inf}) among $\mu\in \mathcal M(\T)$. We then study the more general case of $K$ bounded by a $C(1,\alpha)$ curve $\Gamma$ in Section \ref{C1}, and then by a $C({2,\alpha)}$ curve in Section \ref{C2}. To derive Theorem \ref{keyth} in this setting, for $z\in \Omega$ we make a connection between $\tilde \lambda_{\infty}(\nu,z)$, a modification of $\lambda_{\infty}$ for measures $\nu$ supported on $\Gamma$, with $\lambda_{\infty}(\Phi_{*}\nu,1/\bar{\Phi(z)})$ where $\Phi_{*}\nu$ is the push-forward of $\nu$ on $\T$, $\Phi$ being a conformal map from the exterior of $\Gamma$ to the exterior of $\T$. After some preliminary results, an outline of the proof is given in Section \ref{C1}, followed by the details. We conclude with an interesting observation on the distinction between the cases of $K$ being a curve versus $K$ being an arc.
\\[5pt]
\noindent{\bf Acknowledgements}: We would like to thank the referee for several useful suggestions, including a shorter proof of Theorem 3.2.
\section{General preliminaries}\label{Prelim}

We begin with an elementary result.
\begin{lemma}\label{low-sc}
Let $K$ be a subset of $\C$, $n$ a given positive integer, and $z\in\C$. Then
the map $\nu\in\MM(K)\mapsto B_{n}(\nu,z)\in(0,\infty]$ is weak-$*$ lower semicontinuous.
\end{lemma}
\begin{proof}
By (\ref{B-lambda}), it is equivalent to check that the map $\nu\mapsto\lambda_{n}(\nu,z)$ is upper semicontinuous, which is true since, for each $p\in\PP_{n}$, $p(z)=1$, the map
$$
\nu\mapsto\int_{K}|p|^{2}d\nu
$$
is weak-$*$ continuous, and $\lambda_{n}(\nu,z)$ is obtained as a minimum of such maps.
\end{proof}

We continue with some observations related to \cite{BLOC}. 
\begin{enumerate}
\item The max-min in (\ref{max-min}) coincides with the min-max for general compact $K$, namely
\begin{equation}\label{max-min2}
\max_{\nu\in\MM(K)}\min_{p\in\PP_{n},~p(z_{0})=1}\int_{K}|p|^{2}d\nu=
\min_{p\in\PP_{n},~p(z_{0})=1}\max_{\nu\in\MM(K)}\int_{K}|p|^{2}d\nu.
\end{equation}
This follows from the classical minimax theorem, see the proof of \cite[Proposition 2.1]{BLOC}. 
\item Let $K\subset \mathbb C$ be compact and contain infinitely many points and fix $z_0\not \in K$. For $n\in \mathbb N$, let 
\begin{equation}\label{maxpoly} M_n=M_n(z_0,K):=\sup \{|p(z_0)|: p\in \mathcal P_n, \ ||p||_K\leq 1\}.\end{equation}
There exists a unique $p_n \in \mathcal P_n$ with $||p_n||_K=1$ and $p_n(z_0)=M_n$; in \cite{BLOC} this is called the {\it polynomial of extremal growth relative to $K$ at $z_0$.} Indeed, note that 
$$M_n = \sup \{\frac{|p(z_0)|}{||p||_K}: p\in \mathcal P_n\}= [\inf \{ \frac{||p||_K}{|p(z_0)|}: p\in \mathcal P_n\}]^{-1}$$
and 
\begin{align*}
\inf \{ \frac{||p||_K}{|p(z_0)|}: p\in \mathcal P_n\} & = \inf \{ ||p||_K: p\in \mathcal P_n, \ |p(z_0)|=1\}
\\[5pt]
& =\inf \{ ||1-Q||_K: Q\in \mathcal P_n, \ Q(z_0)=0\}.
\end{align*}
Let $V_n:= \{Q\in \mathcal P_n, \ Q(z_0)=0\}$. This is an $n-$dimensional complex vector space, and 
clearly each nonzero $Q\in V_n$ has at most $n-1$ zeros in $K$ (since $Q(z_0)=0$).  By the classical  Haar uniqueness theorem in Chebyshev approximation (cf., \cite{A}, \cite[Theorem 19]{M}), every 
continuous, complex-valued function on $K$ admits a unique best sup-norm approximant from $V_n$. Applying this to the constant function $1$ there 
exists a unique $Q_n \in V_n$ with $M_n=[||1-Q_n||_K]^{-1}$, and thus $p_n=1-Q_n$. 

\item From 2. and Remark 2.3 in \cite{BLOC}, it follows that the support of an OPM $\nu_n$ of order $n$ for $K$ and $z_0$ as in 2. is contained in 
$$S_n(K):=\{z\in K: |p_n(z)|=||p_n||_K\}.$$
The set $\{z\in \C: |p_n(z)|=||p_n||_K\}$ is a real algebraic curve in $\R^2$ of degree at most $2n$. In particular, for $z_{0}\in\Omega$, the unbounded component of $\C\setminus K$, if $p_n$ is non-constant, any OPM $\nu_{n}$ for $K$ is supported on $\partial\Omega$. A necessary and sufficient condition that $p_n$ be non-constant is that $z_0$ lie outside of 
$$\hat K_n:= \{z\in \C: |q_n(z)|\leq ||q_n||_K \ \hbox{for all} \ q_n \in \mathcal P_n\},$$
the $n-$th order polynomial hull of $K$. Since these sets $\hat K_n$ decrease to
$$\hat K:= \{z\in \C: |q(z)|\leq ||q||_K \ \hbox{for all} \ q \in \bigcup_n \mathcal P_n\},$$
the polynomial hull of $K$, and $\Omega =\C\setminus \hat K$, by appealing to either the Hilbert lemniscate theorem (cf.,\cite[Theorem 5.5.8]{Ra}) or simply Runge's theorem, 
for any $z_0\in \Omega$, there exists $n_0$ sufficiently large so that $p_n$ is non-constant for $n\geq n_0$. 

\end{enumerate}

Thus if, e.g., $K$ is an ellipse of the form 
$$K=\{(x,y)\in \R^2: x^2/a^2 +y^2/b^2 =1\}$$
with $a\not = b$ and $z_0$ lies outside $K$, by Bezout's theorem $S_n(K)$ 
contains at most $4n$ points. Since an OPM $\nu_n$ exists, the support of $\nu_n$ contains at least $n+1$ points. On the other hand, we recall in the next section that for the unit circle $\T=\{(x,y)\in \R^2: x^2 +y^2 =1\}$ and a point $z_0$ with $|z_0|>1$, there exist OPM's $\nu_n$ which are absolutely continuous with respect to arclength measure and hence with support $\T$. It follows from 3. that OPM's for $\bar \D$ and $\T$ coincide. More generally, if $K$ is a compact subset bounded by a $C(2,\alpha)$ curve $\Gamma=\p\Omega$ (as in section 4) and $z_0\in \Omega$, OPM's $\nu_n$ for $K$ and $\Gamma$ coincide, at least for $n$ sufficiently large, which we will always assume.

\section{The unit disk $\D$} 
We begin by recalling that the harmonic measure $\omega_{\D}(z,t)$ for the disk $\D$ and a point $z=|z|e^{i\phi}\in\D$ is given by
\begin{equation}\label{Pois0}
d\omega_{\D}(z,t)=\frac{1-|z|^{2}}{|e^{it}-z|^{2}}\frac{dt}{2\pi}=
\left[\sum_{k=-\infty}^{\infty}\left|z\right|^{|k|} e^{i k(\phi-t)}\right] \frac{dt}{2 \pi}=\Re \bigl(\frac{e^{it} +z}{e^{it} -z}\bigr)\frac{dt}{2 \pi}=:P_z(e^{it})\frac{dt}{2 \pi},
\end{equation}
see e.g.\ \cite[Chapter 4.3]{Ra}. In particular, $d\omega_{\D}(0,t)={dt}/{2 \pi}$. It may also be defined as the balayage $\hat\delta_{z}$ of the Dirac mass $\delta_{z}$ onto the unit circle $\T$, see \cite[Appendix A.3]{ST} or, by conformal invariance, the balayage $\hat\delta_{1/\bar z}$ of $\delta_{1/\bar z}$ onto $\T$.
\begin{definition}
 A positive and finite measure $\mu$ on the unit circle $\T$ satisfies the Szeg\H{o} condition if its density $f=d\mu/d\theta$ satisfies
$$
\int_{\T}\log fd\theta>-\infty.
$$ 
Then, the Szeg\H{o} function is defined by
\begin{equation}\label{def-sz}
D(\mu,z)=\exp\left(\frac{1}{4\pi}\int_{0}^{2\pi}\frac{e^{it}+z}{e^{it}-z}\log f(t)dt\right),\qquad|z|<1.
\end{equation}
\end{definition}
Note that, with $\mu_{a}$ the absolutely continuous part of $\mu$, and $\lambda>0$, one has
\begin{equation}\label{D-pty}
D(\mu,z)=D(\mu_{a},z),\qquad D(\lambda\mu,z)=\sqrt\lambda D(\mu,z).
\end{equation}
It is known, see \cite[Theorem 2.4.1]{S}, that for any measure $\mu$ satisfying the Szeg\H{o} condition,
\begin{equation}\label{asymp-lamb}
\lambda_{\infty}({\mu},z)=(1-|z|^{2})|D(\mu,z)|^{2},\qquad|z|<1.
\end{equation}
We also recall that for any measure $\mu$ on $\T$ the Christoffel function satisfies
\begin{equation}\label{christinv}
|z|^{2n}\lambda_{n}({\mu},z)
=\lambda_{n}({\mu},1/\bar{z}),\qquad z\neq0,
\end{equation}
see e.g.\ \cite[Lemma 2.2.8]{S}. These relations (\ref{asymp-lamb}) and (\ref{christinv}) will be crucial in the sequel, as will the unicity in the next result.

\begin{theorem}\label{optim-D}
Let $z\in\D$.
The unique probability measure $\mu$ on $\T$ that maximizes $\lambda_{\infty}({\mu},z)$,  is the balayage measure $\hat\delta_{z}$.
\end{theorem} 
\begin{proof}
Let $\mu$ be a probability measure on $\T$. By \cite[Theorem 2.7.15]{S}, we have $\lambda_{\infty}({\mu},z)=0$ for any $z\in\D$ precisely when $\mu$ does not belong to the Szeg\H{o} class. Thus, we may assume that $\mu$ belongs to the Szeg\H o class. By (\ref{asymp-lamb}), we are led to maximize $|D(\mu,z)|$. Let $\mu=\mu_{a}+\mu_{s}$, $\mu_{a}=gdt$, $g\in L^{1}(\T)$, be the Radon-Nikodym decomposition of the measure $\mu$. From (\ref{D-pty}), we see that the larger the mass of $\mu_{a}$, the larger the modulus of $D(\mu,z)$. Thus, $\mu_{s}$ should vanish, that is, $\mu$ has to be absolutely continuous. 

We write 
$$d\mu = f(e^{it})d\omega_{\D}(0,t) = \frac{f(e^{it})}{P_z(e^{it})} d\omega_{\D}(z,t).$$ 
Then since
$$\int _{-\pi}^{\pi} \log f(e^{it})d\omega_{\D}(0,t) = \int _{-\pi}^{\pi} \log  \frac{f(e^{it})}{P_z(e^{it})} d\omega_{\D}(z,t)+\int _{-\pi}^{\pi} \log P_z(e^{it})d\omega_{\D}(z,t),$$
it suffices to maximize the entropy functional
$$ \int _{-\pi}^{\pi} \log  \frac{f(e^{it})}{P_z(e^{it})} d\omega_{\D}(z,t)=\int _{-\pi}^{\pi}  \log (\frac{d\mu}{d\omega_{\D}(z,t)}) d\omega_{\D}(z,t)$$
over absolutely continuous probability measures $\mu$. Jensen's inequality yields that the maximum is attained, uniquely, when $d\mu/d\omega_{\D}(z,t)=1$; i.e., $d\mu =d\omega_{\D}(z,t)$.

\end{proof}

It is proved in \cite{BLOC} that, for a given degree $n$ and $z_{0}=|z_{0}|e^{i\phi}$ with $|z_0|>1$, the harmonic measure (\ref{Pois0}) for $1/\bar z_{0}$, 
$$
\mathrm{d} \mu_{P}(\theta):=d\omega_{\D}(1/\bar z_{0},\theta)=\left[\sum_{k=-\infty}^{\infty}\left|z_{0}\right|^{-|k|} e^{i k(\phi-\theta)}\right] \frac{\mathrm{d} \theta}{2 \pi},
$$
is an OPM of order $n$ for $\bar \D$ and $z_0$, as well as any measure $\mu$ whose moments
$$
m_{k}=m_{k}(\mu):=\int_{\T}z^{k}d\mu,\qquad k=0,\pm1,\ldots,\pm n
$$
coincide with those of the harmonic measure:
$$
m_{k}(\mu_{P}):=\int_{\T}z^{k}d\mu_{P}=\int_{\T}z^{k}d \hat \delta_{1/\bar z_{0}}=
\begin{cases}
(\bar z_{0})^{-k}, & k\geq0,
\\[5pt]
z_{0}^{k}, & k<0,
\end{cases}
\qquad |k|=0,1,\ldots, n.
$$
Moreover, from (\ref{max-growth}) and (\ref{maxpoly}), since $M_n(z_0,\T)=|z_0|^n$ we have, for $n\geq0$,
\begin{equation}\label{bnmup}
B_{n}(\mu_{P},z_{0})=|z_{0}|^{2n},\qquad\lambda_{n}(\mu_{P},z_{0})=|z_{0}|^{-2n}.
\end{equation}
\section{Preliminaries when $K$ is bounded by a $C(1,\alpha)$ curve}\label{C1}
Let $K$ be a connected, simply connected, compact subset of $\C$, with nonempty interior.
Let $\Omega$ be the unbounded component of $\C\setminus K$ and $\Omega_{\infty}:=\Omega\cup\{\infty\}$. 
Let $\Phi$ be the conformal map from $\Omega$ to $\C\setminus\D$, with $\Phi(\infty)=\infty$ and $\Phi'(\infty)>0$. 
In this section, we assume that $\Gamma=\p\Omega$ is a $C(1,\alpha)$ curve. 
For $r\geq 1$, we define the level curves of $\Phi$,
$$\Gamma_{r}:=\{z\in\Omega: ~|\Phi(z)|=r\}.$$ 
We will need several results.

We recall a result about sequences of conformal maps, see \cite{W}, Theorem 4 of Section 2.3.

\begin{theorem} \label{walsh} Let $J\subset \C$ be a Jordan curve and
let $D$ be the bounded component of $\C \setminus J$.
Let $\{ D_n \}_{n=1}^\infty$ be a sequence of bounded, 
simply connected domains such that $\overline{D}_{n+1}\subset D_n$ for each $n$ and 
$$
  \bigcap _{n=1} ^\infty D_n  =  \overline{D}.
$$
Given $z_0 \in D$, let
$F,\ F_n$, $n\geq1$, be the conformal
mappings of $D,\ D_n$, $n\geq1$, onto $\D$ which take $z_0$ to the origin and such that $F'(z_0) > 0$ and $F_n'(z_0) > 0$
for each $n$. Then we have
$$
  \lim _{n \to \infty} F_n (z)  =  F(z)  \quad
     {\rm uniformly,\ for\ all\ } z \in \overline{D} .  
$$
\end{theorem}

We will also make use of the Faber polynomials $F_{n}$, $n\geq0$, of the interior 
of $K$, see \cite{Su2}. They are defined by the following identity, see \cite[p.62]{Su2}:
$$
F_{n}(z)=\Phi^{n}(z)+\frac{1}{2 \pi i}\int_{\Gamma_{r}}\frac{\Phi^{n}(t)}{t-z}dt,\qquad |\Phi(z)|>r\geq1,
$$
where we recall that $\Gamma_{r}:=\{z\in\Omega: ~|\Phi(z)|=r\}$.
\begin{proposition}[{\cite[p.61]{Su2}}]\label{Suet-decr}
Let $\Gamma$ be a $C(1,\alpha)$ curve, and let $F_{n}$, $n\geq0$, be the associated Faber polynomials. Let $F$ be a closed subset of the interior of $K$. Then, there is a constant $c(F)$ such that
$$
|F_{n}(t)|\leq\frac{c(F)}{n^{\alpha}},\qquad t\in F.
$$
\end{proposition}
\begin{remark}
The above result on the decrease of Faber polynomials in $K$ also holds for piecewise analytic curves $\Gamma$, see \cite[Theorem 1]{G}.
\end{remark}
\begin{proposition}[{\cite[Theorem 2 p.68]{Su2}}]\label{Fab-asymp}
When $\Gamma$ is a $C(p+1,\alpha)$ curve, $p\in\N$, the following estimate holds:
\begin{equation}\label{Faber-1}
F_{n}(z)=\Phi^{n}(z)
+\OO\left(\frac{\log n}{n^{p+\alpha}}\right)
,\quad n\to\infty,\end{equation}
uniformly for $z\in\bar\Omega$. 
\end{proposition}

We denote by $A(\bar\Om_{\infty})$ the set of functions analytic in a neighborhood of $\bar\Om_{\infty}$.

\begin{proposition}\label{prop-Pad}
Given a function $g \in A(\bar\Om_{\infty})$, $Q_{n}$ any polynomial of degree at most $n$, and $P_{n}$ the unique polynomial of degree at most $n$ such that
\begin{equation}\label{diff-Pad}
Q_{n}(z)g(z)-P_{n}(z)=\OO(z^{-1}),\qquad z\to\infty,
\end{equation}
one has 
\begin{equation}\label{Int-Pade}
Q_{n}(z)g(z)-P_{n}(z)=-\frac{1}{2\pi i}\int_{\Gamma_{g}}Q_{n}(t)\frac{g(t)}{t-z}dt
\end{equation}
for $z$ outside of $\Gamma_{g}$, a simple, positively oriented, curve lying in $K$ and in the domain of analyticity of $g$. 
\end{proposition}
\begin{proof}
Because of (\ref{diff-Pad}), the identity (\ref{Int-Pade}) is a simple consequence of
Cauchy's formula applied to the difference $Q_{n}(z)g(z)-P_{n}(z)$ outside of $\Gamma_{g}$.
\end{proof}
Let $\mu$ be a probability measure on $\Gamma$. We set, for $z\in\Omega$,
\begin{equation} \label{bneqn}
\tilde B_{n}(\mu,z):=\frac{B_{n}(\mu,z)}{|\Phi(z)|^{2n}} \quad \hbox{and} \quad
\tilde\lambda_{n}(\mu,z):=|\Phi(z)|^{2n}\lambda_{n}(\mu,z),\qquad n\geq0,
\end{equation}
and
\begin{equation} \label{binfeqn}
\tilde B_{\infty}(\mu,z):=\limsup_{n\to\infty}\tilde B_{n}(\mu,z)\leq\infty \quad \hbox{and} \quad
\tilde\lambda_{\infty}(\mu,z):=\liminf_{n\to\infty}\tilde\lambda_{n}(\mu,z)\geq0.
\end{equation}
In fact, in Lemma \ref{44} below, we show that the limits exist in (\ref{binfeqn}). Note that since $B_n(\mu,z)$ is weak-* continuous for our class of measures so is $\tilde B_{n}(\mu,z)$.

The idea behind our proof that the weak-* limit of any subsequence $\{\nu_n\}_{n \in Y}, \ Y \subset \N$ of OPM's for $\Gamma$ and $z$ is $\hat \delta_z$, the balayage of the point mass at $z$ to $\Gamma$, is as follows. Using Propositions \ref{Suet-decr}, \ref{Fab-asymp} and \ref{prop-Pad}, we first show in Lemma \ref{44} and Corollary \ref{transport} that for any probability measure $\mu$ on $\Gamma$, $\tilde\lambda_{\infty}(\mu,z)$ (and hence $\tilde B_{\infty}(\mu,z)$) is related to $\lambda_{\infty}(\Phi_{*}\mu,1/\bar{\Phi(z)})$ where $\Phi_{*}\mu\in \mathcal M (\T)$. 
The crux of the matter is to then show that if $\alpha$ is a weak-* limit of a subsequence $\{\nu_n\}_{n \in Y_1}, \ Y_1 \subset Y$ then (a perturbation of) a ``diagonal subsequence'' $\{\tilde B_n(\nu_n,z)\}_{n\in Y_1}$, converges to (a perturbation of) $\tilde B_{\infty}(\alpha,z)$ (Lemma \ref{lem-3-sub}). As in the proof of Lemma \ref{44}, we use Faber polynomials in Lemma \ref{lem-B-ana} as a tool to prove a sort of monotonicity of $\{\tilde B_{n}(\mu,z)\}$ in $n$ for general $\mu$ which is needed to apply Dini's theorem to conclude the proof of Lemma \ref{lem-3-sub}. After the proof of our main result, we make a remark to indicate a  relationship with an underlying general potential-theoretic question.

\begin{lemma} \label{44}
Let $z\in\Om$ and let $\mu$ be a measure on $\Gamma$. We have
\begin{equation}\label{sim}
\tilde\lambda_{\infty}(\mu,z)=\inf\left\{\int_{\Gamma}|f|^{2}d\mu,~f\in A(\bar\Om_{\infty}),~f(z)=1\right\}.
\end{equation}
Moreover, $\tilde\lambda_{\infty}(\mu,z)=\lim_{n\to\infty}\tilde\lambda_{n}(\mu,z)$, i.e., the limit exists.
\end{lemma}
\begin{proof} We first show
$$ \tilde\lambda_{\infty}(\mu,z) \geq \inf\left\{\int_{\Gamma}|f|^{2}d\mu,~f\in A(\bar\Om_{\infty}),~f(z)=1\right\}.$$

Let $\Psi$ be a conformal map from $U=\inter(K)$ to $\D$. We consider the level curves $\tilde\Gamma_{k}:=\{|\Psi|=1-1/k\}, \ k=2,3,...$, and let $\Om_{k}$ be the domain outside of $\tilde\Gamma_{k}$. Then $\Om_k \supset \Om_{k+1} \supset \Om$, and $\Om = \inter(\cap_k \Om_k)$.  Letting $\Phi_{k}$ denote the conformal map from $\Om_{k}$ to $\C\setminus\D$ with $ \Phi_k(\infty)=\infty$ and $\Phi_k'(\infty)>0$, it follows from Theorem \ref{walsh} that $\Phi_{k}$ converges locally uniformly in $\Om$ to $\Phi$.

Fix $k\in\N$. We have, for each $n$,
\begin{align*}
\tilde\lambda_{n}(\mu,z) & =|\Phi(z)|^{2n}
\inf\left\{\int_{\Gamma}|p|^{2}d\mu,~p\in\PP_{n},~p(z)=1\right\}
\\[10pt]
& \geq|\Phi(z)|^{2n}
\inf\left\{\int_{\Gamma}\frac{|p|^{2}}{{(|\Phi_{k}|)}^{2n}}d\mu,~p\in\PP_{n},~p(z)=1\right\}
\\[10pt]
& \geq|\Phi(z)|^{2n} \inf\left\{\int_{\Gamma}|f|^{2}d\mu,~f\in A(\bar\Om_{\infty}),~f(z)=\Phi_{k}(z)^{ -n}\right\}.
\\[10pt]
& =\left|\frac{\Phi(z)}{\Phi_{k}(z)}\right|^{2n} \inf\left\{\int_{\Gamma}|f|^{2}d\mu,~f\in A(\bar\Om_{\infty}),~f(z)=1\right\}.
\end{align*}
In the first inequality, we have used that $|\Phi_{k}|>1$ on $\Gamma$, and in the second inequality, we have used that $p/\Phi_{k}^n$ is analytic in a neighborhood of $\bar\Om_{\infty}$.

Letting $k$ tend to infinity, since $z\in \Omega$ we have $\Phi_{k}(z)\to\Phi(z)$ as $k\to\infty$, and thus
$$
\tilde\lambda_{n}(\mu,z)\geq
\inf\left\{\int_{\Gamma}|f|^{2}d\mu,~f\in A(\bar\Om_{\infty}),~f(z)=1\right\},
$$
which implies the desired inequality:

$$\liminf_{n\to \infty} \tilde\lambda_{n}(\mu,z) =  \tilde\lambda_{\infty}(\mu,z) \geq \inf\left\{\int_{\Gamma}|f|^{2}d\mu,~f\in A(\bar\Om_{\infty}),~f(z)=1\right\}.$$

To show that 
$$ \inf\left\{\int_{\Gamma}|f|^{2}d\mu,~f\in A(\bar\Om_{\infty}),~f(z)=1\right\} 
\geq \limsup_{n\to \infty} \tilde\lambda_{n}(\mu,z), $$
given $\epsilon >0$, take $g \in A(\bar\Om_{\infty})$ with $g(z)=1$ and 
$$ 
 \int_{\Gamma}|g|^{2}d\mu\leq (1+\eps)^{2}\inf\left\{\int_{\Gamma}|f|^{2}d\mu,~f\in A(\bar\Om_{\infty}),~f(z)=1\right\} .
$$
We show that for $n \geq n_0(\epsilon)$, we can find $p_n\in\PP_{n}$ such that
\begin{equation}\label{gapp} 
\int_{\Gamma} {|p_n|^{2}} d\mu  \leq (1+\eps)^{2}\int_{\Gamma}|g|^{2}d\mu 
\quad \hbox{and} \quad p_n(z)=\Phi^{n}(z).
\end{equation}
Applying Proposition \ref{prop-Pad} with the function $g$ and the polynomial $Q_{n}=F_{n}$, the $n-$th Faber polynomial for $K$, and making use of Proposition \ref{Suet-decr}, we conclude that, for some constant $c$ independent of $n$,
$$
\left|F_{n}(t)g(t)-P_{n}(t)\right|\leq\frac{c}{n^{\alpha}},\qquad t\in\Gamma\cup\{z\}.
$$
By Proposition \ref{Fab-asymp} applied with $p=0$, we have $F_{n}\to\Phi^{n}$ on $\Gamma$ uniformly, and $F_{n}(z)\to\Phi^{n}(z)$. Since $|\Phi|=1$ on $\Gamma$ and $g(z)=1$ we get
$$
 |P_{n}|\to |g|,~\text{uniformly on }\Gamma, \quad\hbox{and} \quad P_{n}(z)\to\Phi^{n}(z).
$$
Thus, we get,
for $n\geq n_0(\epsilon)$, that
$p_n:= (\Phi^{n}(z)/P_{n}(z))P_n$ satisfy (\ref{gapp}).
\end{proof}

For a measure $\mu$ on $\Gamma$ we have $\Phi_{*}\mu$ is a measure on the circle $\T$. From (\ref{def-lamb-inf}), 
$$
\lambda_{\infty}(\Phi_{*}\mu,1/\bar{\Phi(z)})=\lim_{n\to\infty}\lambda_{n}(\Phi_{*}\mu,1/\bar{\Phi(z)}).
$$

\begin{corollary}\label{transport}
For any measure $\mu$ on $\Gamma$, it holds that
\begin{equation}\label{link-G-D}
\tilde\lambda_{\infty}(\mu,z)
=\lambda_{\infty}(\Phi_{*}\mu,1/\bar{\Phi(z)}),
\qquad z\in\Omega.
\end{equation}

\end{corollary}
\begin{proof}
One has, in view of (\ref{sim}) and (\ref{christinv}),
\begin{align*}
\tilde\lambda_{\infty} & (\mu,z)
 =\inf\left\{\int_{\Gamma}|f|^{2}d\mu,~f\in A_{e}(\Gamma),~f(z)=1\right\}
\\[10pt]
& =\inf\left\{\int_{\Gamma}|f|^{2}d\Phi_{*}\mu,~f\in A_{e}(\T),~f(\Phi(z))=1\right\}
=\lim_{n\to\infty}|\Phi(z)|^{2n}\lambda_{n}(\Phi_{*}\mu,\Phi(z))
\\[10pt]
& =\lim_{n\to\infty}\lambda_{n}(\Phi_{*}\mu,1/\bar{\Phi(z)})
=\lambda_{\infty}(\Phi_{*}\mu,1/\bar{\Phi(z)}).
\end{align*}
\end{proof}
\section{Case of $K$ bounded by a curve $\Gamma\in C(2,\alpha)$}\label{C2}
With the same notation as section 4 we now assume that $\Gamma=\p\Omega$ is a $C(2,\alpha)$ curve. 
We start with proving a weak monotonicity of the sequence $\{\tilde{B}_{n}(\mu,z)\}_n$ for $\mu$ on $\Gamma$.

\begin{lemma}\label{lem-B-ana}
Let $z\in\Omega$ be fixed. 
Let $\mu$ be any measure supported on $\Gamma$ such that the orthogonal polynomials are well-defined up to degree $N$. Let $n<N$. Then there exist positive numbers $c_{n}\geq1$ such that
\begin{equation}\label{monot}
\tilde{B}_{N-n}(\mu,z)\leq c_{n}\tilde{B}_{N}(\mu,z),\qquad c_{n}=1+\OO\left(\frac{\log n}{n^{1+\alpha}}\right)
,\qquad\text{as }n,N\to\infty,
\end{equation}
where the $c_{n}$'s are independent of the measure $\mu$.
\end{lemma}
\begin{proof}
Let $n<N$. We will prove that
$$
|\Phi(z)|^{2n}\lambda_{N}(\mu,z)\leq c_{n}\lambda_{N-n}(\mu,z)
$$
for appropriate $c_n$, which is equivalent to the inequality in (\ref{monot}).

To estimate $|\Phi(z)|^{2n}\lambda_{N}(\mu,z)$, 
$$|\Phi(z)|^{2n}\lambda_{N}(\mu,z) =|\Phi(z)|^{2n}\min_{p\in{\mathcal P}_{N},~p(z)\neq0}\frac{\int_{\Gamma}|p|^{2}d\mu}{|p(z)|^{2}}$$

$$ \leq|\Phi(z)|^{2n}\min_{p\in{\mathcal P}_{N-n},~p(z)\neq0}\frac{\int_{\Gamma}|F_{n}p|^{2}d\mu}{|F_{n}(z)p(z)|^{2}}$$
$$=\frac{|\Phi(z)|^{2n}}{|F_{n}(z)|^2} \cdot \min_{p\in{\mathcal P}_{N-n},~p(z)\neq0}\frac{\int_{\Gamma}|F_{n}p|^{2}d\mu}{|p(z)|^{2}}.$$

To estimate $\int_{\Gamma}|F_{n}p|^{2}d\mu$, we get
$$\int_{\Gamma}|F_{n}p|^{2}d\mu  \leq \|F_{n}\|^{2}_{\Gamma}\int_{\Gamma}|p|^{2}d\mu$$
so that

$$|\Phi(z)|^{2n}\lambda_{N}(\mu,z) \leq \frac{|\Phi(z)|^{2n}}{|F_{n}(z)|^2} \|F_{n}\|^{2}_{\Gamma} \cdot  \min_{p\in{\mathcal P}_{N-n},~p(z)\neq0}\frac{\int_{\Gamma}|p|^{2}d\mu}{|p(z)|^{2}}.$$
The last minimum 
equals $\lambda_{N-n}(\mu,z)$ and we need estimate 
$$\frac{|\Phi(z)|^{2n}}{|F_{n}(z)|^2} \cdot \|F_{n}\|^{2}_{\Gamma}$$
from above. Using the estimate (\ref{Faber-1}) with $p=1$ for each piece, we obtain
$$
|\Phi(z)|^{2n}\lambda_{N}(\mu,z)\leq \left(1+\OO\left(\frac{\log n}{n^{1+\alpha}}\right)\right)\cdot \lambda_{N-n}(\mu,z),
$$
from which the existence of the $c_{n}$'s follows. The proof shows that they are independent of the measure $\mu$.
\end{proof}
\begin{remark}
In the particular case of $\Gamma=\T$, the unit circle, and $\mu=d\theta/2\pi$, the family $\{z^{n}\}_{n\in\N}$ is an orthonormal basis, and 
$$
B_{n}(\mu,z)=\frac{|z|^{2n+2}-1}{|z|^{2}-1}.
$$
The inequality $\tilde B_{n-1}(\mu,z)\leq\tilde B_{n}(\mu,z)$ is true since it is equivalent to $|z|^{2n+2}-1\leq|z|^{2n+2}$. For the harmonic measure $\mu_P$ in (\ref{Pois0}), from (\ref{bnmup}) we have $B_{n}(\mu_{P},z_{0})=|z_{0}|^{2n}$ so that $ \tilde B_{n}(\mu_{P},z_{0})=1$ for all $n$.
\end{remark}

In Proposition \ref{prop46} and Lemma \ref{lem-3-sub}, the point $z$ is fixed and for any measure $\mu$, we will simply write $\tilde B_{n}(\mu), \tilde B_{\infty}(\mu)$ instead of $\tilde B_{n}(\mu,z), \tilde B_{\infty}(\mu,z)$, and similarly for other expressions depending on $z$.
\begin{proposition} \label{prop46}
Fix $z\in \Omega$.
Assume that a subsequence $\{\nu_{\vphi(n)}\}_n$ of a sequence $\{\nu_{n}\}_n$ of OPM's tends weak-* to a limit measure $\alpha$. 
Then $\alpha$ satisfies the following:\\
1) For all integers $k$, we have
\begin{equation}\label{ineq-B-L}
\tilde B_{k}(\alpha)\leq1\leq 
\tilde\lambda_{k}(\alpha).
\end{equation}
2) $\alpha$ has an infinite number of points in its support.
\end{proposition}
\begin{proof}
To show 1), for a given $k$, we have
\begin{align}
\tilde B_{k}(\alpha) 
\leq
\liminf_{n}\tilde B_{k}(\nu_{\vphi(n)})
\leq\liminf_{n}c_{\vphi(n)-k}\tilde B_{\vphi(n)}(\nu_{\vphi(n)}) 
=\liminf_{n}\tilde B_{\vphi(n)}(\nu_{\vphi(n)}) \leq1. \notag
\end{align}
Here the first inequality 
follows from lower semicontinuity of $B_{k}$ (and hence $\tilde  B_{k}$), recall Lemma \ref{low-sc}. 
The second inequality and the equality use Lemma \ref{lem-B-ana}, while the final inequality uses (\ref{max-growth}) and the fact that $|\Phi(z)|=e^{g_{\Omega}(z)}$, $z\in\Omega$. 
The second inequality in (\ref{ineq-B-L}) is equivalent to the first one.
\\[5pt]
We prove 2) by contradiction. Assume that $\alpha$ has, say, $k$ points in its support. Then, $B_{k}(\alpha)=\infty$, hence $\tilde B_{k}(\alpha)=\infty$, which contradicts the first inequality in (\ref{ineq-B-L}), and proves 2). Note, in particular, that orthogonal polynomials of all degrees are well-defined for the measure~$\alpha$.
\end{proof}
From (\ref{ineq-B-L}), all numbers $\tilde B_{n}(\alpha)$, $n\geq0$, are less than or equal to 1, and thus
$$\tilde B_{\infty}(\alpha)=\lim_{n}\tilde B_{n}(\alpha)\leq1$$
(recall from Lemma \ref{44} the limit exists).
\begin{lemma}\label{lem-3-sub}
Let $\{\nu_{n}\}_{n}$ be a sequence of OPM's on $K$, with $\nu_n$ of order $n$.
For any subsequence $\{\nu_{\vphi_{1}(n)}\}_n$ of $\{\nu_{n}\}_{n}$ with a weak-$^{\star}$ limit $\alpha$, there is a subsequence $\{\nu_{\vphi_{2}(n)}\}_n$ of $\{\nu_{\vphi_{1}(n)}\}_n$ such that
	\begin{equation}\label{conv-unif}
	\lim_{n} \tilde B_{\vphi_{2}(n)}(\nu_{\vphi_{2}(n)})= \tilde B_{\infty}(\alpha).
\end{equation}
\end{lemma}
\begin{proof}
Note that, by Proposition \ref{prop46}, the weak-* convergence $\nu_{\vphi_{1}(n)} \to \alpha$ implies that orthogonal polynomials for the limit measure $\alpha$ exist for any degree $n\geq0$ and 
$$
\tilde B_{\infty}(\alpha)=\lim_{n}\tilde B_{\vphi_{1}(n)}(\alpha).
$$
If the sequence $\{\nu_{\vphi_{1}(n)}\}$ contains an element which appears infinitely many times, then $\alpha$ is equal to this element; hence we may assume that each element in the sequence $\nu_{\vphi_{1}(n)}$ appears at most a finite number of times. For a technical reason in the sequel of the proof (in the definition of the functions $F_{n}$ below), we replace the sequence 
$\nu_{\vphi_{1}(n)}$ with the subsequence, still denoted $\nu_{\vphi_{1}(n)}$, where we keep only the last occurence of each repeated element. Hence, with that change, each element in the sequence $\{\nu_{\vphi_{1}(n)}\}$ appears exactly once.

We choose the subsequence $\{\nu_{\vphi_{2}(n)}\}_n$ of $\{\nu_{\vphi_{1}(n)}\}_n$ in such a way that
\begin{equation}\label{cond}
\forall n\geq1,\quad n\leq 
\vphi_{2}(n)-\vphi_{2}(n-1)<\vphi_{2}(n+1)-\vphi_{2}(n).
\end{equation}
For a measure $\mu$, we set
\begin{equation}\label{C-def}
\tilde C_{n}(\mu) 
:=\left(\prod_{k=0}^{n} c_{\vphi_{2}(k)-\vphi_{2}(k-1)}\right)\tilde B_{\vphi_{2}(n)}(\mu),\qquad n\in\N,
\end{equation}
where the $c_{k}$ are the constants in (\ref{monot}) (recall that they are independent of $\mu$). The sequence $\tilde C_{n}(\mu)$ is increasing with $n$. Indeed, 
$$
\tilde C_{n}(\mu) /\tilde C_{n-1}(\mu) =c_{\vphi_{2}(n)-\vphi_{2}(n-1)}
\tilde B_{\vphi_{2}(n)}(\mu) /\tilde B_{\vphi_{2}(n-1)}(\mu)\geq1,
$$
where the inequality comes from (\ref{monot}).

For the measure $\alpha$ we also define
$$
\tilde C_{\infty}(\alpha)  :=L\tilde B_{\infty}(\alpha),\qquad L:=\prod_{k=0}^{\infty}c_{\vphi_{2}(k)-\vphi_{2}(k-1)}\geq1,
$$
The infinite product in the definition of $L$ converges because of (\ref{cond}) and the asymptotic behavior of the $c_{k}$ as $k$ tends to infinity, see (\ref{monot}). Also, by the choice of the subsequence $\{\nu_{\vphi_{1}(n)}\}_n$, we have
\begin{equation}\label{C-limit}
\tilde C_{\infty}(\alpha)=\lim_{n\to\infty}\tilde C_{n}(\alpha).
\end{equation}
The set of measures $S=\{\nu_{\vphi_{2}(0)},\nu_{\vphi_{2}(1)},\ldots,\alpha\}$ is compact.
Consider the array of values taken by the functions $F_{0}, F_{1},\ldots,F_{n},\ldots,F_{\infty}$ on $S$:
\\[3pt]
\begin{center}
\begin{tabular}{ c|cccccc } 
$F_{\infty}$ & $\tilde C_{\infty}(\alpha)$ & $\tilde C_{\infty}(\alpha)$ & $\tilde C_{\infty}(\alpha)$ & $\tilde C_{\infty}(\alpha)$ & $\tilde C_{\infty}(\alpha)$ &  $\tilde C_{\infty}(\alpha)$ 
\\
$\uparrow$ & $\uparrow$ & $\uparrow$ & $\uparrow$ & $\uparrow$ & $\uparrow$ & $\uparrow$ \\
$F_{n}$ & $\tilde C_{\infty}(\alpha)$ & $\ldots$ & $\tilde C_{\infty}(\alpha)$ & $\tilde C_{n}(\nu_{\vphi_{2}(n)})$ & $\to$ & $\tilde C_{n}(\alpha)$ \\ 
$\vdots$ & $\vdots$ & $\vdots$ & $\vdots$ & $\vdots$ & $\vdots$ & $\vdots$  \\
$F_{1}$ & $\tilde C_{\infty}(\alpha)$ & $\tilde C_{1}(\nu_{\vphi_{2}(1)})$ & $\ldots$ & & $\to$ & $\tilde C_{1}(\alpha)$\\ 
$F_{0}$ & $\tilde C_{0}(\nu_{\vphi_{2}(0)})$ & $\tilde C_{0}(\nu_{\vphi_{2}(1)})$ & \ldots & & $\to$ & $\tilde C_{0}(\alpha)$\\[5pt]
\hline
& $\nu_{\vphi_{2}(0)}$ & $\nu_{\vphi_{2}(1)}$ & $\ldots$ & $\nu_{\vphi_{2}(n)}$ & $\to$ & $\alpha$
\end{tabular}
\end{center}
where all values above the ascending main diagonal $\tilde C_{0}(\nu_{\vphi_{2}(0)}),\tilde C_{1}(\nu_{\vphi_{2}(1)}),\ldots,\tilde C_{n}(\nu_{\vphi_{2}(n)}),\ldots$ are equal to $\tilde C_{\infty}(\alpha)$. Note our choice of the subsequence $\{\nu_{\tilde\vphi_{1}(n)}\}$ insures each $F_n$ is well-defined. 
The following properties are satisfied:\\
a) The function $F_{\infty}$ is constant, hence continuous on $S$. \\
b) For each $n$, $F_{n}$ is continuous at $\alpha$ because $\tilde C_{n}(\nu_{\vphi_{2}(k)})\to\tilde C_{n}(\alpha)$ as $n\leq k\to\infty$. To see this, using (\ref{C-def}) we 
have 
$$\tilde C_{n}(\nu_{\vphi_{2}(k)})
=\left(\prod_{p=0}^{n} c_{\vphi_{2}(p)-\vphi_{2}(p-1)}\right)\tilde B_{\vphi_{2}(n)}(\nu_{\vphi_{2}(k)})
$$
and $\tilde B_{\vphi_{2}(n)}(\nu_{\vphi_{2}(k)})\to \tilde B_{\vphi_{2}(n)}(\alpha)$ as $k\to \infty$ since $\nu_{\vphi_{2}(k)}\to \alpha$ weak-*.
\\
c) At each $\nu_{\vphi_{2}(n)}$, the sequence of functions $F_{0}, F_{1},\ldots,F_{n},\ldots$ increases to $\tilde C_{\infty}(\alpha)$. Indeed, by (\ref{monot}), we have
$$\forall k\leq n-1,~ \tilde C_{k}(\nu_{\vphi_{2}(n)})\leq \tilde C_{k+1}(\nu_{\vphi_{2}(n)}),
\qquad\text{and}\qquad
\tilde C_{n}(\nu_{\vphi_{2}(n)})\leq \tilde C_{n}(\alpha)\leq \tilde C_{\infty}(\alpha),
$$
where the next-to-last inequality uses that $\nu_{\vphi_{2}(n)}$ is an optimal prediction measure.
\\
d) At $\alpha$,  the sequence of functions $F_{0}, F_{1},\ldots,F_{n},\ldots$ also increases to $\tilde C_{\infty}(\alpha)$. This is a consequence of (\ref{monot}) and (\ref{C-limit}).

Hence, from Dini's theorem, we may conclude that the convergence is uniform which implies that 
$\tilde C_{n}(\nu_{\vphi_{2}(n)})\to\tilde C_{\infty}(\alpha)$ and thus also (\ref{conv-unif}).
\end{proof}
\begin{proof}[Proof of Theorem \ref{keyth}]
Let $\{\nu_{\vphi_{1}(n)}\}_n$ be a subsequence of $\{\nu_{n}\}_n$ which converges weak-* to a probability measure $\alpha$. From Lemma \ref{lem-3-sub}, there exists a subsequence 
$\{\nu_{\vphi_{2}(n)}\}_n$ of $\{\nu_{\vphi_{1}(n)}\}_n$
such that, as $n$ tends to infinity,
$$
\tilde B_{\vphi_{2}(n)}(\nu_{\vphi_{2}(n)},z_0)\to\tilde B_{\infty}(\alpha,z_0).
$$
By definition of the OPM's,
$$
\forall\mu\in\MM(\Gamma),~B_{\vphi_{2}(n)}(\nu_{\vphi_{2}(n)},z_0)\leq B_{\vphi_{2}(n)}(\mu,z_0); 
\quad \hbox{hence} \ \tilde B_{\vphi_{2}(n)}(\nu_{\vphi_{2}(n)},z_0)\leq \tilde B_{\vphi_{2}(n)}(\mu,z_0).
$$
Letting $n$ tend to infinity, we get $\tilde B_{\infty}(\alpha,z_0)\leq \tilde B_{\infty}(\mu,z_0)$
which shows that $\alpha$ minimizes $\tilde B_{\infty}(\mu,z_0)$ over $\mu\in\MM(\Gamma)$, or equivalently, maximizes 
$$
\tilde\lambda_{\infty}(\mu,z_{0})=
\liminf_{n\to\infty} 
\tilde\lambda_{n}(\mu,z_{0})
$$
over measures $\mu\in\MM(\Gamma)$. By Corollary \ref{transport}, this is equivalent to the fact that $\Phi_{*}\alpha$ maximizes $\lambda_{\infty}(\nu,1/\bar{\Phi(z)})$ over measures $\nu \in \MM(\T)$. Finally, Theorem \ref{optim-D} shows that 
$$\Phi_{*}\alpha=\hat\delta_{1/\bar{\Phi(z_{0})}}=\hat\delta_{\Phi(z_{0})},$$ 
where the balayage is onto $\T$.
By conformal invariance of the balayage, we obtain that $\alpha$ equals $\hat\delta_{z_{0}}$, the balayage of $\delta_{z_{0}}$ onto~$\Gamma$. 
\end{proof}

We end with a discussion of a related asymptotic problem. For a connected, simply connected, compact subset $K$ of $\C$ we recall from (\ref{max-growth}) that for $z_0\in \Omega$, 
$$B_{n}(\nu_{n},z_{0})=\sup_{p\in\PP_{n}}\frac{|p(z_{0})|^{2}}{\|p\|_{K}^{2}}\leq e^{2ng_{\Omega}(z_0)}=|\Phi(z_0)|^{2n}.$$
In fact, from the first equality together with (\ref{bw}) it follows that
$$\lim_{n\to \infty} \frac{B_{n}(\nu_{n},z_{0})^{1/2n}}{|\Phi(z_0)|}=\lim_{n\to \infty} \tilde B_{n}(\nu_{n},z_{0})^{1/2n}=1.$$
There is the deeper question as to whether the limit of the sequence $\{\tilde B_{n}(\nu_{n},z_{0})\}_n$ -- without the $1/2n$ power -- exists. Clearly 
$$0\leq \liminf_{n\to \infty} \tilde B_{n}(\nu_{n},z_{0}) \leq \limsup_{n\to \infty} \tilde B_{n}(\nu_{n},z_{0}) \leq 1.$$

\begin{enumerate}
\item For the case of the unit circle, since 
$$B_{n}(\nu_{n},z_{0})=\sup_{p\in\PP_{n}}\frac{|p(z_{0})|^{2}}{\|p\|_{K}^{2}}=|z_0|^{2n}=|\Phi(z_0)|^{2n},$$ 
recall (\ref{Pois0}) and (\ref{bnmup}), we have $\tilde B_{n}(\nu_{n},z_{0})=1$ for all $n$. 
\item As a corollary of Lemma \ref{lem-3-sub} and Theorem \ref{keyth}, it follows that {\it for a $C(2,\alpha)$ curve $\Gamma$, we have 
\begin{equation} \label{bnasymp} \lim_{n\to \infty} \tilde B_{n}(\nu_{n},z_{0})=1 \end{equation}
for all $z_0 \in \Omega$.} 
\item For the interval $[-1,1]$, the existence of this limit for $z_0\not \in [-1,1]$ was shown by Yuditskii \cite{Y} and Peherstorfer \cite{Pe}; their proofs are very technical. Writing $\psi(z):=z+\sqrt{z^2-1}$ for the conformal map from $\C \setminus [-1,1]$ onto $\C\setminus \D$, we have  $g_{\Omega}(z)=\log|\psi(z)|$. Two special cases are more easily computed. First, for $x\in \R \setminus [-1,1]$, the polynomial $p_n$ in (\ref{maxpoly}) is the Chebyshev polynomial 
 $$T_n(z)={1\over 2}\left((\psi(z))^n+ ({\psi(z)})^{-n}\right).$$
 Thus for such $x$, from (\ref{max-min2}), 
 $$ \lim_{n\to \infty} \tilde B_{n}(\nu_{n},x)=  \lim_{n\to \infty} \frac{1}{2}
 \frac{|(\psi(x))^n+ ({\psi(x)})^{-n}|}{|\psi(x)|^n}= \frac12.$$
                 Next, for $z=ia, \ a\in \R, \ |a|>1$, from \cite{BLOC} 
                 $$|p_n(ia)|=\sqrt{a^2+1}[|a|+\sqrt{a^2+1}]^{n-1}.$$
                 Since $|\psi(ia)|=|a+\sqrt{a^2+1}|$, we have, for $a>0$,  
                 $$\lim_{n\to \infty} \tilde B_{n}(\nu_{n},ia)=  \lim_{n\to \infty} \frac{\sqrt{a^2+1}[|a|+\sqrt{a^2+1}]^{n-1}}{|a+\sqrt{a^2+1}|^n}=
                 \frac{\sqrt{a^2+1}}{a+ \sqrt{a^2+1}}.$$
                 The results in \cite{Y} and \cite{Pe} seem to indicate that, as with these special cases, for any $z_0\not \in [-1,1]$, 
                 \begin{equation} \label{lessthan1} \lim_{n\to \infty} \tilde B_{n}(\nu_{n},z_0)<1. \end{equation}
                 \item For a circular arc  $A_{\alpha}:=\{z\in \C: |z|=1, \ |\hbox{arg}z|\leq \alpha\}, \ 0<\alpha <\pi$, Eichinger \cite{E} shows that $ \lim_{n\to \infty} \tilde B_{n}(\nu_{n},z_0)$ exists for any $z_0$ with $|z_0|\not = 1$ and he calculates this limit.
\end{enumerate}

Concerning 2., in particular, for the confocal ellipses
$$E_r:=\{ z\in \C: |z-1|+|z+1| = r+1/r\}$$
(\ref{bnasymp}) holds for all points $z_0$ outside $E_r$ for each $r>1$. As $r\to 1$, these ellipses converge to the interval $[-1,1]$, which, according to 3., fails to have this property. We know of no general results on existence of the limit of the sequence $\{\tilde B_{n}(\nu_{n},z_{0})\}_n$. 

\begin{remark} For the interval $[-1,1]$, or, more generally, for a real analytic arc $\gamma$, 
there is a problem with generalizing Lemma \ref{44}, Corollary \ref{transport}, and the ``weak monotonicity'' lemma, Lemma \ref{lem-B-ana}. Indeed, if such results were true for $[-1,1]$, then the proofs of Proposition \ref{prop46} and hence Lemma \ref{lem-3-sub} and Theorem \ref{keyth} would be valid as well. However, equation (\ref{bnasymp}) then gives  $$\lim_{n\to \infty} \tilde B_{n}(\nu_{n},z_{0})=1$$
which contradicts (\ref{lessthan1}). Thus other ideas or techniques are required to deal with arcs.
\end{remark}

\vspace{1cm}
{\obeylines
\texttt{
S. Charpentier, stephane.charpentier.1@univ-amu.fr
F. Wielonsky, franck.wielonsky@univ-amu.fr
Laboratoire I2M - UMR CNRS 7373
Universit\'e Aix-Marseille, CMI 39 Rue Joliot Curie
F-13453 Marseille Cedex 20, FRANCE 
\medskip
N. Levenberg, nlevenbe@indiana.edu
Indiana University, Bloomington, IN 47405 USA
\medskip}
}
\end{document}